\documentclass[reqno]{amsart}
\usepackage{amsmath,amsthm,amsfonts}

\newtheorem{thm}{Theorem}
\newtheorem{lem}[thm]{Lemma}

\newtheorem{prop}[thm]{Proposition}
\theoremstyle{definition}

\newtheorem{rem}[thm]{Remark}

\newcommand{\Z}{\mathbb{Z}}

\newcommand{\R}{\mathbb{R}}

\newcommand{\allone}{\mathbf{1}}

\begin{document}
\title{Large Regular Simplices Contained in a Hypercube with a Common Barycenter}
\author[H. Tamura]{Hiroki Tamura}
\address{Graduate School of Information Sciences,
Tohoku University, Sendai, 980-8579 Japan}
\email{tamura@ims.is.tohoku.ac.jp}
\date{January 14, 2011}

\begin{abstract}
It has been shown that the $n$-dimensional unit hypercube contains an
$n$-dimensional regular simplex of edge length $c\sqrt n$ for arbitrary
$c<1/2$ if $n$ is sufficiently large (Maehara, Ruzsa and Tokushige, 2009).
We prove the same statement holds for some $c>1/2$ even in the special case where
a regular simplex has the same barycenter as that of the unit hypercube.
\end{abstract}

\maketitle

\section{Introduction}
Let $\ell\Delta_n$ be an $n$-dimensional regular simplex of edge length $\ell$,
and let $Q_n=[-1/2,1/2]^n$ be the $n$-dimensional unit hypercube.
In \cite{MRT}, the problem of finding the largest size of an
$n$-dimensional regular which can be contained in the unit hypercube,
is considered.
For a lower bound, they assert that for every $\epsilon_0>0$
there is an $N_0$ such that for every $n>N_0$ one has
$((1-\epsilon_0)/2)\sqrt n\Delta_n\subset Q_n.$
Here, ``$\ell\Delta_n\subset Q_n$'' means that an isometric copy of $\ell\Delta_n$
is contained in $Q_n$.
For the upper bound, if $\ell\Delta_n\subset Q_n$, then $\ell\leq \sqrt{(n+1)/2}$,
and the equality holds if and only if there exists a Hadamard matrix of order $n+1$
\cite{Sch}.

In this paper, we restrict the case where
a regular simplex has the same barycenter as that of the unit hypercube.
In this situation, the above statement for the upper bound
is still valid, and moreover for a lower bound, we improve the above assertion
by inductive construction from Hadamard matrices.
By ``$\ell\Delta_{n,0}\subset Q_n$'' 
, we mean that an isometric copy of $\ell\Delta_n$ 
with barycenter at the origin is contained in $Q_n$. Then we have the following result.
\begin{thm}\label{thm:main}
For every $n$ one has
$$\left(\frac{\sqrt{336}-4-\sqrt 2}{\sqrt{664}}\sqrt n\right)\Delta_{n,0}\subset Q_n.$$
\end{thm}
Note that $(\sqrt{336}-4-\sqrt{2})/\sqrt{664}=0.5012\cdots$,
and thus this theorem improves the result of \cite{MRT}.

\section{Proof of the main result}
We denote the $n\times n$ all-one matrix by $J_n$,
and the all-one vector of length $n$ by $\allone_n$.
For a matrix (or a vector) $A=(a_{ij})$, we define its norm by
$\|A\|=\max_{ij}|a_{ij}|$.

Let
\begin{align*}
f(n)&:=\max\{\ell\mid\ell\Delta_n\subset Q_n\},\quad \text{and}\\
f_0(n)&:=\max\{\ell\mid\ell\Delta_{n,0}\subset Q_n\}.
\end{align*}
Then we have
\begin{equation}
\label{eq:upper}
f_0(n)\leq f(n)\leq \sqrt{\frac{n+1}{2}},
\end{equation}
and $f_0(n)=\sqrt{(n+1)/2}$ holds if and only if
there exists a Hadamard matrix of order $n+1$ \cite{Sch}.

Recall that the barycenter and the circumcenter of a regular simplex coincide,
and the circumradius of $\Delta_n$ is $\sqrt{n/(2n+2)}$.
From this fact, $\sqrt{2}\Delta_{n-1}$ with barycenter at the origin
corresponds to an orthogonal matrix with the following form.
\[
A=\begin{pmatrix}\frac{1}{\sqrt{n}}&p_1\\
\vdots&\vdots\\ \frac{1}{\sqrt{n}}&p_n\end{pmatrix},
\]
where $p_1,\dots,p_n$ are the vertices of $\sqrt{2}\Delta_{n-1}$.
Let $\hat{O}_n$ be the set of all $n\times n$ real orthogonal matrices
with the first column $(1/\sqrt{n})\allone_n^T$.
We denote by $A_1$ the matrix obtained by deleting the first column of $A\in\hat{O}_n$. 
Then we have
$$\hat{O}_n=\{(\frac{1}{\sqrt n}\allone_n^T\;P)\mid 
\text{row vectors of }\frac{1}{\sqrt{2}}P\text{ are the vertices of }\Delta_{n-1,0}\},$$
and $\|A\|=\|A_1\|$ holds for $A\in\hat{O}_n$.
Thus we have the following.
\begin{lem}\label{lem:orth}
$$f_0(n-1)=\frac{1}{\sqrt 2}\max_{A\in\hat{O}_n}\frac{1}{\|A\|}.$$
\end{lem}
\begin{proof}
\begin{align*}
f_0(n-1)&=\max\{\ell\mid\ell\Delta_{n-1,0}\subset Q_{n-1}\}\\
&=\max_{A\in\hat{O}_n}\max\{ \ell \mid \|\frac{\ell}{\sqrt{2}}A_1\|\leq\frac12 \}\\
&=\max_{A\in\hat{O}_n}\frac{1}{\sqrt{2}\|A_1\|}
=\frac{1}{\sqrt{2}}\max_{A\in\hat{O}_n}\frac{1}{\|A\|}.
\end{align*}
\end{proof}

\begin{lem}\label{lem:Fourier}
Let
$$A=\begin{cases}
\sqrt{\frac{2}{n}}\begin{pmatrix}
\frac{1}{\sqrt 2}\allone_n^T&\frac{1}{\sqrt 2}v^T&C&S\end{pmatrix}&(n\text{: even}),\\
\sqrt{\frac{2}{n}}\begin{pmatrix}
\frac{1}{\sqrt 2}\allone_n^T&C&S\end{pmatrix}&(n\text{: odd}),\end{cases}$$
where $C=(c_{ij})$ and $S=(s_{ij})$ are $n\times\lfloor (n-1)/2\rfloor$ matrices
with $c_{ij}=\cos(\theta_j+2ij\pi/n)$ and $s_{ij}=\sin(\theta_j+2ij\pi/n)$ respectively,
and $v=(v_i)$ is the vector of length $n$ with $v_i=(-1)^i$.
Then $A\in\hat{O}_n$ for arbitrary $\theta_j$'s.
\end{lem}
\begin{proof}
Follows from direct calculation of $A^TA$.
\end{proof}
\begin{rem}
Let $T$ be a non-empty subset of $\{1,\dots,\lfloor (n-1)/2\rfloor\}$,
and let $W$ be the submatrix of the above $A$
consisting of the $k$th column of $\sqrt{2/n}C$
and the $k$th column of $\sqrt{2/n}S$ for $k\in T$,
If $n$ is even, let $W'$ be the matrix $W$ with
the column $\sqrt{1/n}v^T$ adjoined.
Then the set of row vectors of $W$ (resp. $W'$) form a spherical 2-design in $\R^{2|T|}$ (resp. $\R^{2|T|+1}$).
The matrices $W$ and $W'$ in the special case where $\theta_j=0$  
for all $j$,
appeared as a construction of spherical $2$-designs in \cite{Mi}.
\end{rem}

\begin{prop}\label{prop:Fourier}
For every $n$, we have $f_0(n)\geq\sqrt{n+1}/2$.
\end{prop}
\begin{proof}
The matrix $A$ given in Lemma~\ref{lem:Fourier} satisfies
$\|A\|\leq\sqrt{2/n}$. Then
Lemma~\ref{lem:orth} implies $f_0(n-1)\geq\sqrt{n}/2$. 
\end{proof}
\begin{rem}
The lower bound in Proposition~\ref{prop:Fourier} is already better
than that of \cite{MRT},
and it can be slightly improved by choosing a matrix $A\in\hat{O}_n$ carefully.
Indeed, $\|A\|$ is minimized when we set
$$\theta_j=\begin{cases}\pi/n&(n\equiv 0\pmod 4)\\
\pi/4&(\text{otherwise})\end{cases}$$
for all $j$,
and thus
$$f_0(n-1)\geq\begin{cases}\displaystyle\frac{\sqrt{n}}{2\cos(\pi/n)}&(n\equiv 0\pmod 4)\\
\displaystyle\frac{\sqrt{n}}{2\cos(\pi/2n)}&(n\equiv 2\pmod 4)\\
\displaystyle\frac{\sqrt{n}}{2\cos(\pi/4n)}&(\text{otherwise}).\end{cases}$$
\end{rem}
To improve the lower bound more for large $n$,
we use inductive construction of orthogonal matrices from Hadamard matrices.

\begin{lem}\label{lem:n-1}
Suppose that $A\in\hat{O}_{n+1}$ has a form
$$A=\begin{pmatrix}\frac{1}{\sqrt{n+1}}&a&u\\
\begin{matrix}\vdots\\ \frac{1}{\sqrt{n+1}}\end{matrix}&v^T&X\end{pmatrix}.$$
Then
$$\widetilde{A}=\begin{pmatrix}\displaystyle
\begin{matrix}\frac{1}{\sqrt n}\\ \vdots\\\frac{1}{\sqrt n}\end{matrix}&
X+(\sqrt{\frac{n}{n+1}}+a)^{-1}(-v^T+\frac{1}{\sqrt{n(n+1)}}\allone_n^T)u\end{pmatrix}\in\hat{O}_n.$$
\end{lem}
\begin{proof}
Since $AA^T=I_{n+1}$, we have
\begin{align*}
a^2+|u|^2&=\frac{n}{n+1},\\
av+uX^T&=-\frac{1}{n+1}\allone_n,\\
v^Tv+XX^T&=I_n-\frac{1}{n+1}J_n,
\end{align*}
and thus we have $\widetilde{A}\widetilde{A}^T=I_n$ by direct calculation.
\end{proof}

From the above lemmas, we have the following.
\begin{prop}\label{prop:f0}
\begin{enumerate}
\item $f_0(2n+1)\geq\sqrt{2}f_0(n)$,
\item $f_0(n-1)>f_0(n)-1/\sqrt{2}$.
\end{enumerate}
\end{prop}
\begin{proof}
(i) Let $B_2=\frac{1}{\sqrt 2}\left(\begin{smallmatrix}1&1\\1&-1\end{smallmatrix}\right)$.
If $A\in\hat{O}_n$,
then $B_2\otimes A\in\hat{O}_{2n}$ and $\|B_2\otimes A\|=\|A\|/\sqrt{2}$.
The result follows from Lemma~\ref{lem:orth}.

(ii) Let $A$ and $\widetilde{A}$ be as in Lemma~\ref{lem:n-1}.
By a suitable change of the
sign of the second column of $A$, we can assume $a\geq 0$.
Since
\begin{align*}
\|\widetilde{A}\|
&=\|X+(\sqrt{\frac{n}{n+1}}+a)^{-1}(-v^T+\frac{1}{\sqrt{n(n+1)}}\allone_n^T)u\|\\
&\leq\|A\|+\sqrt{\frac{n+1}{n}}(\|A\|+\frac{1}{\sqrt{n(n+1)}})\|A\|\\
&=\|A\|(\frac{n+1}{n}+\sqrt{\frac{n+1}{n}}\|A\|),
\end{align*}
and $\|A\|\geq 1/\sqrt{n+1}$, we have
\begin{align*}
&\|\widetilde{A}\|^{-1}-\|A\|^{-1}+1\\
&\geq\frac{1}{\|A\|(\frac{n+1}{n}+\sqrt{\frac{n+1}{n}}\|A\|)}-\frac{1}{\|A\|}+1\\
&=\frac{-\frac{1}{\sqrt{n(n+1)}}+(\sqrt{\frac{n+1}{n}}-1)\|A\|+\|A\|^2}
{\|A\|(\sqrt{\frac{n+1}{n}}+\|A\|)}\\
&=\frac{(-\frac{1}{\sqrt{n+1}}+\|A\|)(\frac{1}{\sqrt n}+\|A\|)
+(\sqrt{\frac{n+1}{n}}-1)(1-\frac{1}{\sqrt{n+1}})\|A\|}
{\|A\|(\sqrt{\frac{n+1}{n}}+\|A\|)}\\
&>0,
\end{align*}
and thus the result follows from Lemma~\ref{lem:orth}.
\end{proof}

\begin{thm}\label{thm:1}
Let $k\geq 1$. If there exists a Hadamard matrix of order $n+k$, we have
$$f_0(n)\geq\frac{\sqrt{n+k}-k+1}{\sqrt 2},$$
and equality holds if and only if $k=1$.
In particular, if there exists a Hadamard matrix of order $4m$ for any $m\leq M$,
then $f_0(n)>(\sqrt{n+4}-3)/\sqrt 2$ holds for any $n\leq 4M-1$.
\end{thm}
\begin{proof}
The former statement is a direct consequence of Proposition~\ref{prop:f0} (ii).
Since $\sqrt{n+k}-k+1$ is decreasing on $k$, the latter statement follows.
\end{proof}

If the Hadamard conjecture is true, we have
$f_0(n)>(\sqrt{n+4}-3)/\sqrt 2$ for any $n$.
This lower bound is close to the upper bound (\ref{eq:upper})
if $n$ is sufficiently large.
Even if we do not assume the Hadamard conjecture, we can estimate a lower bound of $f_0(n)$
by Proposition~\ref{prop:f0}.
\begin{lem}\label{lem:lower}
Let $N$ be a positive integer.
If $f_0(n)\geq c\sqrt{n/2}$ holds for any $n$ with $N\leq n\leq 2N-1$,
then $f_0(n)>(c-(1+\sqrt{2})/\sqrt{N})\sqrt{n/2}$ for any $n\geq N$. 
\end{lem}
\begin{proof}
First, we show that
\begin{equation}\label{eq:ind}
f_0(n)\geq c\sqrt{\frac{n}{2}}-\frac{\sqrt{2}^{k}-1}{2-\sqrt{2}}
\end{equation}
holds for any $n$ and $k$ with $2^kN\leq n\leq 2^{k+1}N-1$, $k\in\Z_{\geq 0}$.
If (\ref{eq:ind}) holds for some $n$ and $k$,
then by Proposition~\ref{prop:f0}, we have
$$f_0(2n)>\sqrt{2}f_0(n)-\frac{1}{\sqrt 2}\geq c\sqrt{n}-\frac{\sqrt{2}^{k+1}-1}{2-\sqrt{2}}$$
and
\begin{align*}
f_0(2n+1)&=\sqrt{2}f_0(n)\geq c\sqrt{n}-\frac{\sqrt{2}^{k+1}-\sqrt{2}}{2-\sqrt{2}}\\
&=c\sqrt{\frac{2n+1}{2}}-\frac{\sqrt{2}^{k+1}-1}{2-\sqrt{2}}
+c(\sqrt{n}-\sqrt{\frac{2n+1}{2}})+\frac{1}{\sqrt 2}.
\end{align*}
Since $c\leq\sqrt{(n+1)/n}$ by (\ref{eq:upper}), we can derive
$c(\sqrt{n}-\sqrt{(2n+1)/2})+1/\sqrt 2>0$ for $n\geq 1$.
Thus, by induction on $k$, (\ref{eq:ind}) holds
for any $n$ and $k$ with
$2^kN\leq n\leq 2^{k+1}N-1$, $k\in\Z_{\geq 0}$. Therefore
\begin{align*}
\sqrt{\frac{2}{n}}f_0(n)
&\geq c-\frac{1}{\sqrt n}\frac{\sqrt{2}^{k}-1}{\sqrt{2}-1}\\
&>c-\frac{1}{\sqrt{2^kN}}\frac{\sqrt{2}^{k}}{\sqrt{2}-1}\\
&=c-\frac{1+\sqrt{2}}{\sqrt{N}}.
\end{align*}
\end{proof}

\begin{proof}[Proof of Theorem~\ref{thm:main}]
The smallest order for which no Hadamard matrix is presently known is 668
\cite{CRC,KT}.
Thus $f_0(n)>(\sqrt{n+4}-3)/\sqrt{2}$ holds for any $n\leq 663$ by Theorem~\ref{thm:1}.
Since $(\sqrt{n+4}-3)/\sqrt{n}$ is a increasing function,
we can set $N=332$ and $c=(\sqrt{336}-3)/\sqrt{332}$ in Lemma~\ref{lem:lower}.
Then we have
$$\frac{f_0(n)}{\sqrt{n}}>\frac{c}{\sqrt 2}-\frac{(1+\sqrt{2})}{\sqrt{2N}}
=\frac{\sqrt{336}-4-\sqrt{2}}{\sqrt{664}}$$
for $n\geq 332$.
For $n<332$, we have
$f_0(n)/\sqrt{n}\geq\max\{\sqrt{n+1}/2\sqrt{n},(\sqrt{n+4}-3)/\sqrt{2n}\}$
by Proposition~\ref{prop:Fourier} and Theorem~\ref{thm:1}.
This lower bound exceeds $(\sqrt{336}-4-\sqrt{2})/\sqrt{664}$. 
\end{proof}


\begin{thebibliography}{9}
\bibitem{CRC} R. Craigen and H. Kharaghani,
Hadamard matrices and Hadamard designs, in: Handbook of Combinatorial Designs
(C. J. Colbourn and J. H. Dinitz, eds.), Second Edition,
pp. 273--280, Chapman \& Hall/CRC Press, Boca Raton, FL, 2007.
\bibitem{KT} H. Kharaghani and B. Tayfeh-Rezaie,
A Hadamard matrix of order 428,
J. Combin. Des. 13 (2005), 435--440.
\bibitem{MRT} H. Maehara, I. Z. Ruzsa and N. Tokushige,
Large regular simplices contained in a hypercube,
Period. Math. Hungarica 58 (2009), 121--126.
\bibitem{Mi} Y. Mimura,
A construction of spherical 2-design,
Graphs Combin. 6 (1990), 369--372.
\bibitem{Sch} I. J. Schoenberg,
Regular simplices and quadratic forms,
J. Lond. Math. Soc. 12 (1937), 48--55.
\end{thebibliography}
\end{document}